\documentclass[12pt,reqno]{amsart}

\usepackage[english]{babel}
\usepackage[utf8]{inputenc}
\usepackage{amsmath, amsthm, amssymb}
\usepackage{amsmath}
\usepackage{amsfonts}
\usepackage{amsthm}
\usepackage{mathrsfs}
\usepackage{amsrefs}
\usepackage{hyperref}
\usepackage{graphicx,epstopdf,color}
\usepackage{appendix}

\usepackage{xcolor}

\renewcommand{\leq}{\leqslant}
\renewcommand{\le}{\leqslant}
\renewcommand{\geq}{\geqslant}
\renewcommand{\ge}{\geqslant}
\renewcommand{\mathcal}{\mathscr}

\theoremstyle{plain}
\newtheorem{theorem}{Theorem}[section]
\newtheorem{corollary}[theorem]{Corollary}

\newtheorem{conjecture}[theorem]{Conjecture}

\theoremstyle{definition}
\newtheorem{definition}[theorem]{Definition}
\newtheorem{remark}[theorem]{Remark}

\numberwithin{equation}{section}

\newcommand{\R}{\mathbb{R}}
\newcommand{\e}{\varepsilon}
\newcommand{\abs}[1]{\left\lvert#1\right\rvert}
\newcommand{\norm}[1]{\|{#1}\|}

\begin{document}
	
\title{(Non)local $\Gamma$-convergence}

\author[Serena Dipierro]{Serena Dipierro}
\address{S.D., Department of Mathematics and Statistics,
	University of Western Australia,
	35 Stirling Highway, WA6009 Crawley, Australia}
\email{serena.dipierro@uwa.edu.au}

\author[Pietro Miraglio]{Pietro Miraglio}
\address{P.M., Dipartimento di Matematica, Universit\`a degli Studi di Milano, Via Cesare Saldini 50,
	20133 Milan, Italy, Departament de Matem\`{a}tica Aplicada I,
	Universitat Polit\`{e}cnica de Catalunya, Diagonal 647, 08028 Barcelona, Spain}
\email{pietro.miraglio@unimi.it}

\author[Enrico Valdinoci]{Enrico Valdinoci}
\address{E.V., Department of Mathematics and Statistics,
	University of Western Australia,
	35 Stirling Highway, WA6009 Crawley, Australia}
\email{enrico.valdinoci@uwa.edu.au}

\begin{abstract}
We present some long-range interaction models for phase coexistence
which have recently appeared in the literature, recalling also their relation
to classical interface and capillarity problems. In this note, the main focus
will be on the $\Gamma$-convergence methods, emphasizing similarities
and differences between the classical theory and the new trends of investigation.

In doing so, we also obtain some new, more precise $\Gamma$-convergence
results in terms of ``interior'' and ``exterior'' contributions. We also
discuss the structural differences between $\Gamma$-limits
and ``pointwise'' limits, especially concerning the ``boundary terms''.
\end{abstract}

\maketitle

\section{Introduction}

The goal of this note is to present and discuss some recent developments
in the mathematical analysis
of phase separation models, with special attention to some problems described
in terms of long-range particle interactions, and exploiting methods
and techniques related to the classical
notion of $\Gamma$-convergence.
\medskip

In 1975 De Giorgi and Franzoni~\cite{DG,DGF} introduced the notion of $\Gamma$-convergence as a new type of convergence for functionals, particularly suitable for the study of variational problems. 
This new tool quickly became popular in the calculus of variations, as it allows one
to relate a sequence of minimization problems depending on a parameter (that can be discrete or continuous) with a limit problem, that can
possibly have a different nature from the original problems, in terms
of energy functionals,
functional spaces, physical modelization, etc.
In spite of the structural differences
between the original functionals and the limit one, this kind of convergence preserves
the notion of
minimizers in the limit, hence suggesting some relations between the limit problem and the sequence of functionals taken into account.
\medskip
 
We now recall one of the possible definitions of $\Gamma$-convergence,
referring to the monographs~\cite{B,DM} for a complete
introduction to the subject of $\Gamma$-convergence and for all
the equivalent definitions of this notion.

Given a family of functionals $F_j$ defined on the function spaces $X_j$, we are interested in the minimization problems
\[
\min\left\{F_j(u) : \,\,u\in X_j	\right\},
\]
depending on a parameter $j$, and we want to relate this sequence with a limit problem, of the form
\[
\min\left\{F(u) : \,\,u\in X	\right\}.
\] 
\begin{definition}\label{def_gamma}
	We say that $F_j$ converges in the $\Gamma$-sense to $F$ if the two following conditions are satisfied:
\begin{itemize}
	\item[(i)] for every $u\in X$ and every sequence~$u_j$ converging\footnote{Here and in the following we take $X$ such that $X_j\subseteq X$ for every $j$ and we define $F_j\equiv+\infty$ in~$X\setminus X_j$.} to $u$ in $X$, it holds that
	\begin{equation*}
	\liminf_{j\to\infty} F_j(u_j)\geq F(u);
	\end{equation*}
	\item[(ii)] for every $u\in X$ there exists a sequence~$u_j\in X$ converging to $u$ in $X$ such that
	\begin{equation*}
	\limsup_{j\to\infty} F_j(u_j)\leq F(u).
	\end{equation*}
\end{itemize}
\end{definition}

These two conditions can be understood by analogy with the direct
method of the calculus of variations, keeping in mind that here
we have a sequence of functionals instead of a single one.
Indeed, on the one hand
condition $(\textnormal{i})$ plays the role of the lower
semicontinuity, providing a lower bound for the sequence of minimizers.
On the other hand, condition~$(\textnormal{ii})$ is an upper bound that ensures the
optimality of the limit functional $F$ among all the ones satisfying condition $(\textnormal{i})$. Assuming that an equi-coerciveness condition is satisfied by the sequence of functionals $F_j$, a minimizing sequence $(\overline{u}_j)$ for the family $F_j$ converges to a function $\overline{u}\in X$. Whenever $F_j$ satisfies also $(\textnormal{i})$ and $(\textnormal{ii})$, we then have that
\begin{itemize}
	\item there exists a minimizer $\overline{u}$ of the limit functional $F$ defined on $X$;
	\item the sequence of minimizers $\overline{u}_j$ of $F_j$ converges in $X$ to $\overline{u}$;
	\item the sequence of minima $F_j(\overline{u}_j)$ converges to $F(\overline{u})$.
\end{itemize}
These three properties make the $\Gamma$-convergence a very useful tool in the study of minimum problems arising in the calculus of variations. 

In particular, given an energy functional depending on a parameter, we can relate it to a new minimum problem by taking its $\Gamma$-limit for the parameter going to infinity. This limit problem contains somehow the relevant features of the original one, as its minimizers are the limits of sequences of
minimizers of the original variational problem. In this way, through the study of the minimizers of the limit functional, one can recover some important information about the original problem. 
\medskip

The rest of this note is organized as follows. In the forthcoming
Section~\ref{SEZ2} we recall one of the first examples
of $\Gamma$-convergence, also motivated by the theory
of phase coexistence. Then, in Section~\ref{GABmwe}
we discuss some capillarity problems focused at detecting
suitable boundary effects.

In Section~\ref{JNujHHSN09393} we present
some long-range interaction models
describing nonlocal phase separation, nonlocal
capillarity and water waves problems. In this section
we also provide some new results about the ``interior''
and ``exterior'' $\Gamma$-convergence of nonlocal energy
functionals. Finally, in Section~\ref{JA:AKMZ09wjjjdjd}
we briefly recall some $\Gamma$-convergence results
in the fractional parameter and we compare the
notions of $\Gamma$-convergence and ``pointwise''
limits, stressing important differences with respect to the
boundary contributions obtained via these two alternative
approaches.

\section{$\Gamma$-convergence results for the classical phase
coexistence energy functional}\label{SEZ2}

A paradigmatic example of $\Gamma$-convergence 
is provided by some classical results for
the Allen-Cahn, or Cahn-Hilliard, energy functional,
which models the separation of the two phases of a fluid in a container.

In 1958, Cahn and Hilliard~\cite{CH} proposed a new model for a
two-phase fluid in a container, in which the phase transition occurs
continuously in a thin layer, instead of discontinuously along an interface.  
The model is closely related to the
minimization of the Helmholtz
free energy in a liquid-gas system, as originally
proposed by J. D. van der Waals~\cite{WAALS} --- see also~\cite{ALLEN}. 

In this
model, one assumes that the configurations of the fluid in a container $\Omega\subset\R^3$ are described by a mass density $u$ that takes values in $[-1,1]$, the pure phases being $A:=\{u=-1\}$ and $B:=\{u=1\}$. Then, the energy associated with the configuration of the fluid is the sum of a potential term, in which a nonnegative
double-well\footnote{
We say that a function $W:\R\to[0,+\infty)$ is a ``double-well'' with zeros in $\pm1$ if it satisfies
\[
W\in C^2(\R),\,\,W(\pm1)=0,\,\,W>0\,\text{in}\,\,\R\setminus\{\pm1\},\,\,W'(\pm1)=0,\,\,W''(\pm1)>0.
\]
} function~$W$ vanishing at $-1$ and $1$ appears, and a Dirichlet term, that
penalizes the transitions from one phase to the other. That is, the energy associated to a configuration $u$ is
\begin{equation}\label{energy_mm}
\widetilde F_\e(u,\Omega):=\e^2\int_\Omega \abs{\nabla u(x)}^2\,dx+\int_\Omega W(u(x))\,dx,
\end{equation}
with the parameter $\e$ being representative of the thickness of the layer where the phase transition occurs. In particular, since this length is supposed to be much smaller than the size of the container $\Omega$, it is interesting to study the asymptotic behavior of the configuration, i.e., its limit as $\e\to0^+$. 
\medskip

This type of analysis was initiated
in the sense of $\Gamma$-convergence 
by Modica and Mortola~\cite{M,MM}, who considered a suitable rescaling of the energy $\widetilde F_\e$.
Namely, they took into account the functional
$$ F_\e(u,\Omega):=\frac1\e\,\widetilde F_\e(u,\Omega)=\e
\int_\Omega \abs{\nabla u(x)}^2\,dx+\frac1\e\int_\Omega W(u(x))\,dx,$$
and proved that it $\Gamma$-converges to 
\begin{equation}\label{energy_limit}
F(u,\Omega):=
\begin{cases}
\begin{aligned}
&c_*\,\textnormal{Per}(E,\Omega) \qquad &\text{if}\,\,u_{|\Omega}=\chi_E-\chi_{\mathcal{C}E},\,\,\text{for some set}\,\,E\subset\Omega,
\\
&+\infty\qquad&\text{otherwise},
\end{aligned}
\end{cases}
\end{equation}
where~$c_*>0$ is a normalization
constant depending only on $n$ and $W$,
and~$\textnormal{Per}(\cdot,\Omega)$ represents the perimeter functional inside the set $\Omega$. 

As a consequence, as $\e\to 0^+$,
the minimizers of the functional $F_\e$ converge to the
minimal surfaces, i.e., the minimizers of the perimeter functional. 
We refer to the books~\cite{G, MAGGI} for a complete introduction to the theory of minimal surfaces
and the notion of perimeter. 
\medskip

In particular, the theory of~$\Gamma$-convergence
of phase transitions to minimal surfaces has a geometric counterpart
in the convergence of the level sets of the minimizers of~$F_\e$. More precisely,
as established in~\cite{CORDO}, if~$u$ is a minimizer of
\[
F_1(u,\Omega):=\int_\Omega \abs{\nabla u(x)}^2\,dx+\int_\Omega W(u(x))\,dx,
\]
and~$u_\e(x):=u(x/\e)$, then, up to a subsequence,
for every~$\vartheta\in(0,1)$, the set~$\{u_\e\in(-\vartheta,\vartheta)\}$
converges locally uniformly as $\e\to0^+$ to~$\partial E$, being~$E$ a local minimizer
of the perimeter functional. That is, for any~$R > 0$ and any~$\delta> 0$
there exists~$\e_0\in (0, 1)$, possibly depending on~$R$ and~$\delta$,
such that, if~$\e\in(0,\e_0]$ then
\begin{equation}\label{COCP}\{u_\e\in(-\vartheta,\vartheta)\}\cap B_R\;
\subseteq\;
\bigcup_{p\in\partial E}B_\delta(p).\end{equation}
The proof of~\eqref{COCP} in~\cite{CORDO}
relies on suitable energy and density estimates.
More specifically, it is proved in~\cite{CORDO}
that if~$u$ is a minimizer of~$F_1$ in~$B_{R+1}$
with~$R>1$, then
\begin{equation}\label{COCP2}
F_1 (u,B_R)\le CR^{n-1},
\end{equation}
for some constant $C>0$. Also, if~$\vartheta_1$, $\vartheta_2\in(-1,1)$
and~$ u$ is a minimizer of~$F_1$ in~$B_R$ with~$u(0)>\vartheta_1$,
then there exist~$R_o(\vartheta_1,\vartheta_2)>1$ and~$c_o>0$ such
that, for all~$R\ge R_o(\vartheta_1,\vartheta_2)$,
\begin{equation}\label{COCP3} |\{u>\vartheta_2\}\cap B_R|\ge c_o\, R^n.\end{equation}
That is, according to~\eqref{COCP2},
the energy of the minimizers ``mostly
arise from a codimension $1$ interface'', and, in light of~\eqref{COCP3},
unless the solution at a given point (say the origin) is very close to
a pure phase, we have that the two phases in a large ball occupy
a measure which is comparable to the one of the ball itself (i.e., no phase gets lost,
at least in a measure theoretic sense).
\medskip

A very strong connection between phase transition models and minimal
surfaces is highlighted
by a celebrated conjecture of E. De Giorgi~\cite{CONG}	
about the rigidity properties
of monotone solutions to the Allen-Cahn equation $\Delta u = W'(u)$
in $\R^n$,
which can be formulated as follows:

\begin{conjecture} \label{CDG}
Let~$u\in C^2(\R^n)\cap L^\infty(\R^n)$ be
a solution of
$$ \Delta u(x) = W'(u(x))\qquad{\mbox{ for all }}x\in\R^n,$$
and assume also that
$$ \frac{\partial u}{\partial x_n}(x) >0\qquad{\mbox{ for all }}x\in\R^n.$$
Then, is it true that $u$ depends only on one Euclidean variable
(i.e., there exist~$u_0:\R\to\R$ and~$\omega\in S^{n-1}$
such that~$u(x)=u_0(\omega\cdot x)$ for all $x\in\R^n$),
at least if~$n\le8$?
\end{conjecture}

Conjecture~\ref{CDG} gave rise to several papers about the rigidity
of the solutions to the Allen-Cahn equation. We refer to the survey~\cite{FV}
for an introduction to this line of research.

\section{Boundary effects and capillarity problems}\label{GABmwe}
 
In the two-phase model in~\eqref{energy_mm}
the boundary contact energy is assumed to be
negligible, since the model mainly focuses on the formation of the 
phase interfaces inside the domain. 
In order to quantitatively
take into account the boundary effects of the domain
on the phase separation, several other models have been
designed.
\medskip

As a matter of fact, to understand the influence of boundary
effects, a classical model is the one describing ``capillarity'' phenomena
in a water-drop problem, in which
the boundary contact energy between the fluid and the 
wall becomes nonnegligible. 
In this case, the model considers
a liquid droplet of constrained mass occupying a small region~$E$ in a
container~$\Omega$, and the energy functional associated to~$E$ is of the form
\begin{equation}\label{local_capillarity}
G(E):=\textnormal{Per}(E,\Omega)+\sigma\,
\textnormal{Per}(E,\partial\Omega),
\end{equation}
where $\sigma\in[-1,1]$ is the
``relative adhesion coefficient'', that measures the liquid-wall tension with respect to the liquid-air tension. 
See Chapter~19 in~\cite{MAGGI} and the references therein for
a thorough presentation of classical droplet and capillarity problems.\medskip

The functional $G$ in~\eqref{local_capillarity}
shares some obvious similarities with
the functional~$F$ in~\eqref{energy_limit}, 
and therefore, in light of the discussion in Section~\ref{SEZ2},
it is natural to ask whether $G$ can be seen as the $\Gamma$-limit of some
modification of the
phase interface energy functional in~\eqref{energy_mm}. 

In~\cite{M1}, Modica established the $\Gamma$-convergence of the energy
\begin{equation}\label{boundary_en1}
G_\e(u,\Omega):=\e\int_\Omega \abs{\nabla u(x)}^2\,dx+\frac1\e\int_\Omega W(u(x)
)\,dx+\int_{\partial\Omega}V(u(x))\,d\mathcal{H}^{n-1}(x),
\end{equation} 
where $V$ is a nonnegative continuous function, not necessarily of double-well type,
and~$\mathcal{H}^{n-1}$ denotes the $(n-1)$-dimensional Hausdorff measure.
Specifically,
in~\cite{M1} it is proved that, for problems
with prescribed mass, $G_\e$ converges in the $\Gamma$-sense to
the capillarity energy functional $G$
in~\eqref{local_capillarity}.
The relative adhesion coefficient~$\sigma$ appearing in the $\Gamma$-limit~\eqref{local_capillarity} depends only on~$W$ and~$V$, and is explicitly computed in~\cite[Theorem~2.1]{M1}.
\medskip

A modification of the energy $G_\e$ defined in~\eqref{boundary_en1} was considered
in~\cite{ABS,ABS1} by Alberti, Bouchitté, and Seppecher, consisting of the energy functional
\begin{equation}\label{boundary_en_abs}
\mathcal{G}_\e(u,\Omega):=\e\int_\Omega \abs{\nabla u(x)}^2\,dx
+\frac1\e\int_\Omega W(u(x))\,dx+\lambda_\e\int_{\partial\Omega}V(
u(x))\,d\mathcal{H}^{n-1}.
\end{equation} 
Here, $W$ is still a double-well potential vanishing in $\pm1$, while $V$ --- contrary to~\eqref{boundary_en1} --- is a double-well potential vanishing in $\alpha$ and $\beta$, and $\lambda_\e$ is a parameter that goes to infinity when $\e\to0^+$, satisfying
\begin{equation}\label{kappa}
\lim_{\e\to0^+}\e\log\lambda_\e=k \qquad \text{with}\,\,k\in(0,+\infty).
\end{equation}
Under these assumptions --- which are different in the energy boundary term with respect to~\cite{M1} ---
it is established in~\cite{ABS}
that
the energy functional $\mathcal{G}_\e$ 
$\Gamma$-converges
to the limit energy
\begin{equation*}
\mathcal{G}(u):=
\begin{cases}
\begin{aligned}
&\inf\left\{\phi(u,v):\,\,v\in BV(\partial\Omega,\{\alpha;\beta\})	\right\}	\qquad 		&\text{if}\,\,\,u\in BV(\Omega,\{-1;1\}),
\\
&+\infty 		\qquad 		&\text{otherwise},
\end{aligned}
\end{cases}
\end{equation*}
where for every $u\in BV(\Omega,\{-1;1\})$ and $v\in BV(\partial\Omega,\{\alpha;\beta\})$ the function $\phi(u,v)$ is defined as
\begin{equation}\label{phi_funct}
\phi(u,v):=\mathcal{H}^{n-1}(Su)
+\sigma\int_{\partial\Omega}\abs{H(Tu)-H(v)}\,d\mathcal{H}^{n-1}
+c\,\mathcal{H}^{n-2}(Sv).
\end{equation}
Here, $Tu$ denotes the trace of $u$ on the boundary of $\Omega$, $H$ is the primitive function of~$2\sqrt{W}$, while the parameters $\sigma$ and $c$ depend only on $W$, $V$, and $k$, and are explicitly defined in~\cite{ABS}.

Also, in~\eqref{phi_funct}, with $Su$
we denote the
set of the points in which $u$ is essentially\footnote{One says that~$u$
is essentially continuous at a point~$x$
if
for every~$\varepsilon>0$ there exists~$\delta>0$
such that for almost all~$y$, $z\in B_\delta(x)$ one has that~$|f(y)-f(z)|<\varepsilon$.}
discontinuous. In this setting,
if $u\in BV(\Omega,\{-1;1\})$, then $\mathcal{H}^{n-1}(Su)$ is the measure of the interface between the pure phases $\{u=1\}$ and $\{u=-1\}$. It is well-known that a function $u$ that belongs to $\{-1;1\}$ almost everywhere has bounded variation if and only if the measure of the jump-set $Su$ is finite. 

Similarly, if $v\in BV(\partial\Omega,\{\alpha;\beta\})$, then $\mathcal{H}^{n-2}(Sv)$ denotes the $(n-2)$-dimensional measure of the interface between the boundary phases $\{v=\alpha\}$ and $\{v=\beta\}$. Finally, the second term in the definition \eqref{phi_funct} evaluates the energy of the transition from $Tu$ to $v$ that occurs on the boundary. 

The energy functional $\mathcal{G}$
is introduced in~\cite{ABS} as a relaxation\footnote{The relaxation procedure outlined in~\cite{ABS} is necessary as the capillarity functional $\mathcal{G}_\sharp$ in~\eqref{line tension}
is not semi-continuous, and this leads to minimum problems which are not well-posed.} of
a capillarity functional with line tension energy, which can be seen as a modification of the functional~$G$ defined in~\eqref{local_capillarity}. If we take $\alpha=-1$ and $\beta=1$, then the capillarity functional with line tension is \begin{equation}\label{line tension}
\mathcal{G}_\sharp(E):=\mathcal{H}^{n-1}(\Omega\cap\partial E)+\sigma\,\mathcal{H}^{n-1}(\partial\Omega\cap\partial E)+c\,\mathcal{H}^{n-2}\left((\overline{\partial E\cap\Omega})\cap\partial\Omega\right),
\end{equation}
where $E:=\{u(x)=1\}$. In the three-dimensional case, the so-called ``line tension energy'', which is the last term in~\eqref{line tension}, models an energy concentrated along the line $(\overline{\partial E\cap\Omega})\cap\partial\Omega$ where the interface liquid-air $\partial E\cap\Omega$ meets the boundary~$\partial\Omega$ of the container.
\medskip

The results in~\cite{ABS} have later been extended in~\cite{Gon} to the functional
\begin{equation}\label{MdM}\begin{split}
\mathcal{G}^a_\e(u,\Omega)
\,&:=\e^{1-a}\int_\Omega \abs{\nabla u(x)}^2h^a(x) \,dx+\frac1{\e^{1-a}}\int_\Omega W(u(x))h^{-a}(x)
\,dx
\\
&\hspace{2cm}+\lambda_\e\int_{\partial\Omega}V(u(x))\,d\mathcal{H}^{n-1}(x),\end{split}
\end{equation}
where $a\in(-1,0)$, $h:\Omega\to\R$ is the distance function to the boundary of $\Omega$, and $\lambda_\e\to+\infty$ as $\e\to0^+$ with some specific behavior, different from the one in \eqref{kappa}. More precisely, in~\cite{Gon} 
it is proved that the energy functional $\mathcal{G}^a_\e$ achieves the same $\Gamma$-limit for every $a\in(-1,0)$ as the one attained by $\mathcal{G}_\e$ defined in~\eqref{boundary_en_abs}.

\section{Local and nonlocal contributions in the $\Gamma$-limit}\label{JNujHHSN09393}

In this section, we
describe some phase separation models in which the interaction
energy is of nonlocal type.
For this, we start by presenting
the results in~\cite{ABS1}, focusing on the dimension~$n=1$. In~\cite{ABS1}, the authors 
consider an interval $I\subset\R$
and the energy
functional
\begin{equation}\label{abs1}
\mathcal{G}^1_\e(v):=\e\iint_{I\times I}\frac{\abs{v(x)-v(y)}^2}{\abs{x-y}^2}\,dx\,dy
+\lambda_\e\int_{I}W(v(x))\,dx,
\end{equation}
where $W$ is a double-well potential with zeros in $-1$ and $1$,
and $\lambda_\e$ is a positive parameter depending on $\e$ and satisfying~\eqref{kappa}. 

Then, the main result in~\cite{ABS1} establishes that the energy functional defined in~\eqref{abs1} $\Gamma$-converges in the $L^1$-topology to
\begin{equation}\label{abs1_limit}
\mathcal{G}^{1}(v):=
\begin{cases}
\begin{aligned}
&8k\,\mathcal{H}^0(Sv) \qquad &
\text{if}\,\,v\in BV(I,\{-1;1\}),
\\
&+\infty \qquad &\text{otherwise},
\end{aligned}
\end{cases}
\end{equation}
where $k$ is the one in~\eqref{kappa} and $Sv$ is the set of the points in which $v$ is essentially discontinuous.
Since we are assuming $n=1$, this simply means that at those points the function is discontinuous with the left-hand limit being different from the right-hand limit.

As customary, $\mathcal{H}^0$ denotes the 0-dimensional Hausdorff measure, corresponding to the ``counting measure''
(hence, $\mathcal{H}^0(Sv)$ is simply the ``number of jumps''
of the step function~$v$), and $BV(I,\{-1;1\})$ the space of the functions with bounded variation which are defined on $I\subset\R$ with values in $\{-1;1\}$ almost everywhere.
\medskip

In the context of the $\Gamma$-convergence of the
functional $\mathcal{G}_\e^a$ defined in~\eqref{MdM}, the study of the $\Gamma$-limit of an interaction
energy in dimension $n=1$
was addressed in~\cite{Gon} for $a\in(-1,0)$, corresponding to the fractional
parameter $s\in(1/2,1)$. Specifically, for an interval $I\subset\R$, in~\cite{Gon} the author considers the energy
\begin{equation}\label{AG}
\mathcal{G}^{1,a}_\e(v):=\e^{1-a}\iint_{I\times I
}\frac{\abs{v(x)-v(y)}^2}{\abs{x-y}^{1+2s}}\,dx\,dy+\lambda_\e\int_{I}V(v(x))\,dx,
\end{equation}
where $1-a=2s$, proving that it $\Gamma$-converges to $\mathcal{G}^1(v)$ defined in~\eqref{abs1_limit}.
\medskip

It is interesting to remark that the models presented in~\eqref{abs1}
and~\eqref{AG}, though of nonlocal nature, converge to a $\Gamma$-limit, namely
the one in~\eqref{abs1_limit}, which is local and classical.

In the following pages, we present other long-range interaction models for phase transitions and discuss their $\Gamma$-limits. Interestingly, the $\Gamma$-limits of the following functionals reduce to local limit problems for suitable ranges of a fractional parameter (corresponding to ``weakly nonlocal'' interactions), but conserves the nonlocal feature of the original problem for other ranges of this parameter (corresponding to ``strongly nonlocal'' interactions). Observe that the behavior in the strongly nonlocal regime represents a novelty with respect to the previous works~\cite{ABS1,Gon}, in which this range of parameters was not considered.

More precisely,
in~\cite{SVzero, SV, SVdue} Savin and the third author study the $\Gamma$-convergence,
as well as the geometric convergence of level sets of the minimizers,
for $\e\to0^+$ of a proper rescaling of the interaction energy
\begin{equation}\label{ENE} \begin{split}J_\e(u,\Omega)\,&:=
\e^{2s}\iint_{\Omega\times\Omega }\frac{\abs{u(x)-u(y)}^2}{\abs{x-y}^{1+2s}}\,dx\,dy
\\
&\hspace{1.5cm}+2\e^{2s}\int_{\Omega\times\mathcal{C}\Omega }\frac{\abs{u(x)-u(y)}^2}{
\abs{x-y}^{1+2s}}\,dx\,dy+\int_{\Omega}W(u(x))\,dx,
\end{split}\end{equation}
where $s$ is a parameter in $(0,1)$, $W$ a double-well potential, and $\Omega\subset\R^n$ a bounded domain whose complement is $\mathcal{C}\Omega:=\R^n\setminus\Omega$.

In order to describe the result in detail, we introduce the setting in~\cite{SV}. We let~$X:=\{u\in L^\infty(\R^n): \norm{u}_{L^\infty(\R^n)}\leq1\}$ and we say that a sequence~$u_j\in X$ converges to $u$ in $X$ if $u_j$ converges to $u$ in $L^1_\textnormal{loc}(\R^n)$.

The energy considered in~\cite{SV} can be seen as a suitable nonlocal analogue of the classical
model in~\eqref{energy_mm}. Indeed, in~\eqref{ENE} the classical Dirichlet-type energy is replaced by
a long-range interaction
energy consisting of the $\Omega$-contribution in the $H^s$-seminorm of~$u$. 
In the classical case, only local interactions
count in the Dirichlet energy, and the state of the fluid outside the container is not taken into account. 
In this new long-range setting, it is assumed that every particle interacts
with all the other ones,
inside and outside of the container, carrying a smaller contribution as the distance
between the particle increases (and the energy
functional in~\eqref{ENE} takes into account
all the particle interactions in which at least one of the particles
lies in the container).\medskip

In particular, we define the ``interior contribution'' as
\begin{equation*}
\mathcal{K}^{int}(u,\Omega):=\iint_{\Omega\times\Omega}\frac{\abs{u(x)-u(y)}^2}{\abs{x-y}^{n+2s}}\,dx\,dy,
\end{equation*}
and the ``exterior contribution'' as
\begin{equation*}
\mathcal{K}^{ext}(u,\Omega):=2\iint_{\Omega\times\mathcal{C}\Omega}\frac{\abs{u(x)-u(y)}^2}{\abs{x-y}^{n+2s}}\,dx\,dy.
\end{equation*}
Then, we set
\begin{equation*}
\mathcal{K}(u,\Omega):=\mathcal{K}^{int}(u,\Omega)+\mathcal{K}^{ext}(u,\Omega).
\end{equation*}
We observe that in this type
of energy functionals
we omit the contributions for $(x,y)\in\mathcal{C}\Omega\times\mathcal{C}\Omega$, since we are interested in variational problems in which all the admissible competitors are fixed outside of~$\Omega$.

Then, the energy in~\eqref{ENE} can be written as
\[
J_\e(u,\Omega)=\e^{2s}\mathcal{K}(u,\Omega)+\int_\Omega W(u(x))\,dx.
\]
In order to obtain a relevant $\Gamma$-limit, in~\cite{SV} a proper rescaling of the energy $J_\e$ is taken into account. In the present work, we make
this rescaling more explicit, by also
highlighting the different contributions coming from the interior and the exterior parts of the energy. For this, we define
\begin{equation}\label{local_part}
\mathcal{F}^{int}_\e(u,\Omega):=
\begin{cases}
\begin{aligned}
&\mathcal{K}^{int}(u,\Omega)+\frac{\e^{-2s}}{2}\int_\Omega W(u(x))\,dx
\quad &\text{if}\,\,\,s\in\bigg(0,\frac12\bigg),
\\
&\abs{\log\e}^{-1}\mathcal{K}^{int}(u,\Omega)+\frac{\abs{\e\log\e}^{-1}}{2}\int_\Omega W(u(x))\,dx
\quad &\text{if}\,\,\,s=\frac12,
\\
&\e^{2s-1}\mathcal{K}^{int}(u,\Omega)+\frac{\e^{-1}}{2}\int_\Omega W(u(x))\,dx
\quad &\text{if}\,\,\,s\in\bigg(\frac12,1\bigg);
\end{aligned}
\end{cases}
\end{equation}
and
\begin{equation}\label{nonlocal_part}
\mathcal{F}^{ext}_\e(u,\Omega):=
\begin{cases}
\begin{aligned}
&\mathcal{K}^{ext}(u,\Omega)+\frac{\e^{-2s}}{2}\int_\Omega W(u(x))\,dx
\quad &\text{if}\,\,\,s\in\bigg(0,\frac12\bigg),
\\
&\abs{\log\e}^{-1}\mathcal{K}^{ext}(u,\Omega)+\frac{\abs{\e\log\e}^{-1}}{2}\int_\Omega W(u(x))\,dx
\quad &\text{if}\,\,\,s=\frac12,
\\
&\e^{2s-1}\mathcal{K}^{ext}(u,\Omega)+\frac{\e^{-1}}{2}\int_\Omega W(u(x))\,dx
\quad &\text{if}\,\,\,s\in\bigg(\frac12,1\bigg).
\end{aligned}
\end{cases}
\end{equation}
The sum of $\mathcal{F}^{int}_\e$ and $\mathcal{F}^{ext}_\e$ is the object of the $\Gamma$-convergence result in~\cite{SV}, i.e.,
one defines
\begin{equation*}
\mathcal{F}_\e(u,\Omega):=\mathcal{F}^{int}_\e(u,\Omega)+\mathcal{F}^{ext}_\e(u,\Omega).
\end{equation*}

The $\e$-rescaling in the definitions of $\mathcal{F}^{int}_\e$ and $\mathcal{F}^{ext}_\e$ can be seen as a convenient one in order to obtain a significant $\Gamma$-limit. It is worth observing that for the case $s=1/2$ the $\e$-weights in the definitions of $\mathcal{F}^{int}_\e$ and $\mathcal{F}^{ext}_\e$ satisfy\footnote{
This follows from the fact that $$\lim_{\e\to0^+}\abs{{\log\e}}^{-1}\log\left(\e^{-1}\abs{\log\e}^{-1}\right)=1.$$
} the limit assumption~\eqref{kappa} with $k=1$ that is taken in~\cite{ABS1}.

In order to define the $\Gamma$-limit of the energy functionals studied in~\cite{SV},
we recall the notion of
fractional perimeter, as introduced
in~\cite{CRS}. Given two measurable and disjoint sets $E, F\subset\R^n$, one defines
\begin{equation*}
I_s(E,F):=\iint_{E\times F}\frac{dx\,dy}{\abs{x-y}^{n+2s}},
\end{equation*}
where $s\in(0,1/2)$.
Then, we define the ``interior contribution'' of the fractional perimeter as 
\begin{equation}\label{iwdjc:01}
\textnormal{Per}_s^{int}(E,\Omega):=I_s(E\cap\Omega,\Omega\setminus E)
\end{equation}
and the ``exterior contribution'' as
\begin{equation}\label{iwdjc:02}
\textnormal{Per}_s^{ext}(E,\Omega):=I_s(E\cap\mathcal{C}\Omega,\Omega\setminus E) +I_s(E\cap\Omega,\mathcal{C}\Omega\cap\mathcal{C}E).
\end{equation}
Finally, the full fractional perimeter of a set $E$ in $\Omega$ is defined as
\[
\textnormal{Per}_s(E,\Omega):=\textnormal{Per}_s^{int}(E,\Omega)+\textnormal{Per}_s^{ext}(E,\Omega).
\]
In this setting, the $\Gamma$-limit functional $\mathcal{F}$ in~\cite{SV} is as follows:
\begin{equation}\label{ogujkvmn}
\mathcal{F}(u,\Omega):=
\begin{cases}
\begin{aligned}
&\textnormal{Per}_s(E,\Omega) \qquad &\text{if}\,\,s\in\left(0,1/2\right)\,\,\,\text{and}\,\,\,u_{|\Omega}=\chi_E-\chi_{\mathcal{C}E},
\\
&c_*\,\textnormal{Per}(E,\Omega) \qquad &\text{if}\,\,s\in[1/2,1)\,\,\,\text{and}\,\,\,u_{|\Omega}=\chi_E-\chi_{\mathcal{C}E},
\\
&+\infty\qquad&\text{otherwise},
\end{aligned}
\end{cases}
\end{equation}
where $c_*$ is a constant depending only on $n$, $s$, and $W$, which is explicitly determined in~\cite{SV}.

For further reference, it is also convenient,
in the case $s\in(0,1/2)$, to
reformulate and extend
the $\Gamma$-convergence result in~\cite{SV} in terms
of ``interior'' and ``exterior'' limit functionals:

\begin{theorem}
	\label{thm_nonlocal_gamma}
	Let $s\in(0,1/2)$ and $\Omega\subset\R^n$ be a bounded domain. Then, 
	\begin{itemize}
	\item[(a)] $\mathcal{F}^{int}_\e$ $\Gamma$-converges to the interior contribution in the fractional perimeter, i.e., 
	\begin{equation*}
	\mathcal{F}^{int}(u,\Omega):=
	\begin{cases}
	\begin{aligned}
	&\textnormal{Per}_s^{int}(E,\Omega) \qquad &\text{if}\,\,u_{|\Omega}=\chi_E-\chi_{\mathcal{C}E},
	\\
	&+\infty\qquad&\text{otherwise}.
	\end{aligned}
	\end{cases}
	\end{equation*}
	\item[(b)]$\mathcal{F}^{ext}_\e$ $\Gamma$-converges to the exterior contribution in the fractional perimeter, i.e., 
	\begin{equation*}
	\mathcal{F}^{ext}(u,\Omega):=
	\begin{cases}
	\begin{aligned}
	&\textnormal{Per}_s^{ext}(E,\Omega) \qquad &\text{if}\,\,u_{|\Omega}=\chi_E-\chi_{\mathcal{C}E},
	\\
	&+\infty\qquad&\text{otherwise}.
	\end{aligned}
	\end{cases}
	\end{equation*}
	\item[(c)]$\mathcal{F}_\e$ $\Gamma$-converges to the functional $\mathcal{F}$
defined in~\eqref{ogujkvmn}.	\end{itemize}
\end{theorem}

\begin{proof}[Sketch of the proof]
	Since the same strategy works for all three cases, let us deal with point (a). One observes that
	\begin{equation}\label{equiv}
	\mathcal{F}^{int}_\e(u,\Omega)=\mathcal{K}^{int}(u,\Omega)=\mathcal{F}^{int}(u,\Omega) \qquad \text{if}\,\,\,u_{|\Omega}=\chi_E-\chi_{\mathcal{C}E}.
	\end{equation}
	First, we want to prove point (i) in Definition~\ref{def_gamma}. For every sequence~$u_\e$ converging to $u$ in $X$, we can assume that
	\[
	\liminf_{\e\to0^+}\mathcal{F}_\e^{int}(u_\e,\Omega)=l<+\infty,
	\]
	otherwise the claim is trivial. Taking a subsequence~$u_{\e_k}$ attaining the above limit and a further subsequence (that we still name $u_{\e_k}$) converging almost everywhere to~$u$, we deduce that
	\[
	l=\lim_{k\to\infty}\mathcal{F}_{\e_k}^{int}(u_{\e_k},\Omega)\geq\lim_{k\to\infty}\frac{1}{2\e^{2s}_k}\int_\Omega W(u_{\e_k}(x))\,dx.
	\] 
	Therefore, the integral of $W(u)$ over $\Omega$ is zero at the limit and we deduce that $u(x)\in\{-1;1\}$ for almost every $x\in\Omega$, that is $u_{|\Omega}=\chi_E-\chi_{\mathcal{C}E}$ for some set $E\subset\R^n$. Now, by Fatou's lemma and~\eqref{equiv} we have
	\[
	\liminf_{\e\to0^+} \mathcal{F}^{int}_\e(u_\e,\Omega)\geq \mathcal{F}^{int}(u,\Omega),
	\]
	which is the desired inequality.
	
Then, to prove point (ii) in Definition~\ref{def_gamma}, we assume that~$u_{|\Omega}=\chi_E-\chi_{\mathcal{C}E}$ for some set $E\subset\R^n$, otherwise the claim is trivial. Then, by taking a constant sequence~$u_\e:=u$ and using~\eqref{equiv}, it follows that
	\[
	\limsup_{\e\to0^+} \mathcal{F}^{int}_\e(u_\e,\Omega)\leq \mathcal{F}^{int}(u,\Omega),
	\]
	concluding the proof of point (a).
\end{proof}

As a consequence of Theorem~\ref{thm_nonlocal_gamma}, we can obtain a new
result about the $\Gamma$-convergence to a nonlocal capillarity
functional. Indeed,
a fractional analogue of the capillarity functional $G$ defined
in~\eqref{local_capillarity} is studied in~\cite{MaggiV, dmv2}.
For a bounded container $\Omega\subset\R^n$ and for every set $E\subset\Omega$, one
takes into account the energy functional
\begin{equation*}
\mathcal{E}_s(E,\Omega):=I_s(E,\Omega\setminus E)+\sigma I_s(E,\mathcal{C}\Omega),
\end{equation*}
where $\sigma$ is the relative adhesion coefficient that we introduced for the classical capillarity energy --- see~\eqref{local_capillarity}. For every $s\in(0,1/2)$ we define the energy
\begin{equation*}
\mathcal{J}_{\e,s}(u,\Omega):=\mathcal{K}^{int}(u,\Omega)+\sigma\mathcal{K}^{ext}(u,\Omega)+\e^{-2s}\int_\Omega W(u(x))\,dx.
\end{equation*}
Then, we have the following result for the $\Gamma$-convergence of the energy $\mathcal{J}_{\e,s}$.
\begin{corollary}
	Let $s\in(0,1/2)$ and $\Omega$ be a bounded domain. Then, $\mathcal{J}_{\e,s}$ converges in the $\Gamma$-sense to the fractional capillarity energy defined as
	\begin{equation*}
	\mathcal{J}_s(E,\Omega):=
	\begin{cases}
	\begin{aligned}	
	&\mathcal{E}_s(E,\Omega) \qquad &\text{if}\,\,u_{|\Omega}=\chi_E-\chi_{\mathcal{C}E},
	\\
	&+\infty\qquad&\text{otherwise}.
	\end{aligned}
	\end{cases}
	\end{equation*}
\end{corollary}
\begin{proof}
	The result follows from Theorem~\ref{thm_nonlocal_gamma}, the subadditivity of the $\limsup$, and the superadditivity of the $\liminf$.
\end{proof}

Now we focus instead on the case $s\in[1/2,1)$. In this setting, and using the tools in~\cite{SV}, we can prove a $\Gamma$-convergence result for the functionals~$\mathcal{F}^{int}_\e$ and~$\mathcal{F}_\e$ which is similar to, but slightly stronger than, the claim in~\cite[Theorem 1.4]{SV}. We state it in the following theorem and we then sketch its proof, which is obtained by adapting the arguments in~\cite{SV}.

\begin{theorem}  
	\label{thm_local_gamma}
	Let $s\in[1/2,1)$ and $\Omega\subset\R^n$ be a bounded domain with Lipschitz boundary. Then, for any $u\in X$,
	\begin{itemize}
		\item[(i)] for every $u_\e$ that converges to $u$ in $X$,
	\[
	\liminf_{\e\to0^+}\mathcal{F}^{int}_\e(u_\e,\Omega)\geq \mathcal{F}(u,\Omega);
	\]
		\item[(ii)] there exists $u_\e$ that converges to $u$ in $X$ and such that 
		\[
		\limsup_{\e\to0^+}\mathcal{F}_\e(u_\e,\Omega)\leq \mathcal{F}(u,\Omega).
		\]
	\end{itemize}
\end{theorem}

\begin{proof}[Sketch of the proof]
We start with the proof of point (i). We recall that in~\cite{SV} 
it is proved that,
	for every $u_\e$ converging to $u$ in $X$,
	\begin{equation}\label{6TRA}
	\liminf_{\e\to0^+}\mathcal{F}_\e(u_\e,\Omega)\geq \mathcal{F}(u,\Omega).
	\end{equation}
	Actually, the proof in~\cite{SV} 
can be adapted to show point (i) in Theorem~\ref{thm_local_gamma},
which is slightly stronger than~\eqref{6TRA}, as $\mathcal{F}_\e(u_\e,\Omega)\geq\mathcal{F}_\e^{int}(u_\e,\Omega)$. 
	
To prove point~(i), we can assume that 
\begin{equation}\label{hp_liminf}
	\liminf_{\e\to0^+}\mathcal{F}^{int}_\e(u_\e,\Omega)=l<+\infty,
\end{equation}
	otherwise the claim in~(i) is trivial. {F}rom \eqref{hp_liminf}, it follows the existence of a subsequence of $u_\e$, that we still name $u_\e$, such that $u_\e$ converges to $\chi_E-\chi_{\mathcal{C}E}$ in~$L^1(\Omega)$ for some set $E\subset\R^n$ with finite perimeter in $\Omega$. 	
	This is proved in~\cite[Proposition 3.3]{SV}, under the hypothesis that the~$\liminf$ of~$\mathcal{F}_\e(u_\e,\Omega)$ is finite. However, one can weaken this hypothesis
	and assume~\eqref{hp_liminf} instead, from which one can deduce that~$\mathcal{F}^{int}_\e(u_\e,\Omega)$ is uniformly bounded, by eventually passing to a subsequence, and carry out the whole proof.
	
	Since $E$ has finite perimeter in $\Omega$, by classical results in Geometric Measure Theory --- see~\cite[Theorem 4.4]{G} --- we have
	\[
	\textnormal{Per}(E,\Omega)=\mathcal{H}^{n-1}\left(\partial^*E\cap\Omega\right),
	\]
	where $\partial^*E$ is the ``reduced boundary'' of the set $E$. We refer again to~\cite{G,MAGGI} for the theory of sets with finite perimeter and in particular for the definition of the reduced boundary. 
	Then, by the rectifiability of the reduced boundary, for every $\alpha>0$ we can find a collection of balls $B_j$ with radii $\rho_j>0$, whose smallness depends from~$\alpha$, such that
	\[
	\textnormal{Per}(E,\Omega)\leq\alpha+\omega_{n-1}\sum_{j=0}^{+\infty}\rho_j^{n-1},
	\]
	where $\omega_{n-1}$ is the measure of the $(n-1)$-dimensional unit ball.
	By Vitali's covering theorem we can assume that these balls are disjoint, hence
	\[
	\mathcal{F}^{int}_\e(u_\e,\Omega)\geq\sum_{j=0}^{+\infty}\mathcal{F}^{int}_\e(u_\e,B_j).
	\]
	Now, the $\liminf$ of the functional $\mathcal{F}^{int}_\e$ can be explicitly estimated in case the domain is a ball. Indeed, we can use Proposition 4.3 in~\cite{SV} that states\footnote{Observe that Proposition 4.3 in~\cite{SV} is stated for $\mathcal{F}^{int}_\e$ as in~\eqref{4.3}, not for $\mathcal{F}_\e$.
	} that
	\begin{equation}\label{4.3}
	\liminf_{\e\to0^+}\mathcal{F}^{int}_\e(u_\e,B_\rho)\geq\omega_{n-1}\rho^{n-1}\left(c_*-\eta(\alpha)\right),
	\end{equation}
	with $\eta(\alpha)\to0^+$ as $\alpha\to0^+$ and $c^*$ being the constant appearing in the definition of~$\mathcal{F}$. Combining the above results we deduce that	
	\[
	\liminf_{\e\to0^+}\mathcal{F}^{int}_\e(u_\e,\Omega)
	\geq \omega_{n-1}\left(c_*-\eta(\alpha)\right)\sum_{j=0}^{+\infty}\rho_j^{n-1}
	\geq\left(c_*-\eta(\alpha)\right)\left(\textnormal{Per}(E,\Omega)-\alpha\right),
	\]
	and letting $\alpha\to0^+$ we prove point (i).
	\medskip
	
The proof of point (ii) relies on the recovery sequence constructed
in Proposition~4.6 of~\cite{SV}.
\end{proof}

{F}rom Theorem~\ref{thm_local_gamma} we easily observe that the two functionals $\mathcal{F}^{int}_\e$ and $\mathcal{F}_\e$ attain the same $\Gamma$-limit when $s\in[1/2,1)$. Indeed, since $\mathcal{F}^{int}_\e(u,\Omega)\leq\mathcal{F}_\e(u,\Omega)$ for every function $u$ and domain $\Omega$, from Theorem~\ref{thm_local_gamma} we deduce that for any $u\in X$
\begin{itemize}
	\item[(iii)] for every $u_\e$ that converges to $u$ in $X$,
	\[
	\liminf_{\e\to0^+}\mathcal{F}_\e(u_\e,\Omega)\geq \mathcal{F}(u,\Omega);
	\]
	\item[(iv)] there exists $u_\e$ that converges to $u$ in $X$ and such that 
	\[
	\limsup_{\e\to0^+}\mathcal{F}^{int}_\e(u_\e,\Omega)\leq \mathcal{F}(u,\Omega).
	\]
\end{itemize}
That is, combining (i), (ii), (iii), and (iv), both $\mathcal{F}_\e^{int}$ and $\mathcal{F}_\e$ converge to the $\Gamma$-limit~$\mathcal{F}$, that for $s\in[1/2,1)$ is defined as
\begin{equation*}
\mathcal{F}(u,\Omega):=
\begin{cases}
\begin{aligned}
&c_*\, \textnormal{Per}(E,\Omega) \qquad &\text{if}\,\,\,u_{|\Omega}=\chi_E-\chi_{\mathcal{C}E},
\\
&+\infty\qquad&\text{otherwise}.
\end{aligned}
\end{cases}
\end{equation*}
\begin{remark}\label{rmk_boundary}
	The phenomena highlighted in~\cite{SV}
	emphasizes a structural
	difference between the strongly nonlocal regime, i.e., when $s\in(0,1/2)$,
	and the weakly nonlocal one in which $s\in[1/2,1)$. 

This difference also affects the different behavior of the interior
and exterior contributions of the energy functional in the $\Gamma$-limit.
Indeed, in the case~$s\in(0,1/2)$
Theorem~\ref{thm_nonlocal_gamma} shows that
both the interior and the exterior components
of the fractional
phase coexistence functional $\mathcal{F}_\e$ 
converge to two different and nontrivial $\Gamma$-limits,
whose sum is the full fractional perimeter of a set $E$ in a domain $\Omega$.

On the other hand, when $s\in[1/2,1)$,
the nonlocal interactions on $\Omega\times\mathcal{C}\Omega$
in the functional $\mathcal{F}_\e$ disappear in the $\Gamma$-limit.
As a matter of fact, since
$$ 0\le \mathcal{F}^{ext}_\e(u_\e,\Omega)=\mathcal{F}_\e(u_\e,\Omega)
-\mathcal{F}^{int}_\e(u_\e,\Omega),$$
we have that
\begin{eqnarray*}&& 0\le \limsup_{\e\to0^+}\mathcal{F}^{ext}_\e(u_\e,\Omega)=
\limsup_{\e\to0^+}\Big(\mathcal{F}_\e(u_\e,\Omega)
-\mathcal{F}^{int}_\e(u_\e,\Omega)\Big)\\&&\qquad\qquad\le
\limsup_{\e\to0^+}\mathcal{F}_\e(u_\e,\Omega)-
\liminf_{\e\to0^+}\mathcal{F}^{int}_\e(u_\e,\Omega).\end{eqnarray*}
Thus, by Theorem~\ref{thm_local_gamma} and assuming $s\in[1/2,1)$, we have that for every $u\in X$ there exists a sequence $u_\e$ converging to~$u$ in~$L^1_\textnormal{loc}(\R^n)$ such that
$$ \limsup_{\e\to0^+}\mathcal{F}^{ext}_\e(u_\e,\Omega)=0.$$
\end{remark}

\medskip

We recall that the convergence of the level sets of the minimizers
described in~\eqref{COCP}
possesses a natural nonlocal counterpart, as established in~\cite{SVzero, SVdue}.
More precisely, the statement in~\eqref{COCP} holds true
for the rescaled minimizers of~$\mathcal{F}_1$ , defined as
\[
\mathcal{F}_1(u,\Omega):=\mathcal{K}(u,\Omega)+\int_\Omega W(u(x))\,dx.
\]
The only difference with the setting in~\eqref{COCP} is that
the limit set~$E$ is now a local minimizer for the classical perimeter when~$s\in[1/2,1)$,
and a local minimizer for the nonlocal perimeter when~$s\in(0,1/2)$.

The geometric convergence proofs in~\cite{SVzero, SVdue}
also rely on energy and density estimates which can be seen as a nonlocal
counterpart of the classical ones in~\eqref{COCP2} and~\eqref{COCP3}.
More precisely, while~\eqref{COCP3} holds the same in the nonlocal
case (i.e., phases do not get lost in the measure theoretic sense),
the nonlocal counterpart of~\eqref{COCP2} takes into account different
scaling properties depending on the nonlocal exponent. Namely,
if~$u$ is a minimizer of~$\mathcal{F}_1$ in~$B_{R+1}$
with~$R>2$, then
\begin{equation}\label{8uJSnsnn2yedhhdhdh}
\mathcal{F}_1 (u,B_R)\le \begin{cases}
CR^{n-2s} & \quad{\mbox{ if }}s\in(0,1/2),\\
CR^{n-1}\log R& \quad{\mbox{ if }}s=1/2,\\
CR^{n-1} & \quad{\mbox{ if }}s\in[1/2,1),
\end{cases}\end{equation}
for some~$C>0$ depending on $n$, $s$, and $W$.

That is, comparing~\eqref{COCP2} and~\eqref{8uJSnsnn2yedhhdhdh},
the energy of the nonlocal minimizers still behaves as if the interfaces were flat,
but in this case the energy contribution in a large ball has a ``faster'' growth
due to the strongly long-range interaction arising when~$s\in(0,1/2]$.
For further details on the one-dimensional case, see also~\cite{PALAT}.
\medskip

We also mention that
the results and the techniques in~\cite{SV} have been used by the second and
the third author in~\cite{MV} to study the $\Gamma$-convergence of a nonlocal functional
arising in a model for water waves (see also~\cite{DMV} for a detailed
presentation of the physical models).
The energy functional related to this problem
depends on a parameter $s\in(0,1)$ and can be described as follows.
One defines
$$ S_s(\xi)=\frac{J_{1-s}(-i|\xi|)}{J_{s-1}(-i|\xi|)}\,|\xi|^{2s},$$
where $J_k$ is the Bessel function of the first kind of order~$k$.
In this setting, $S_s$ plays the role of a ``Fourier multiplier'',
and it has an interesting algebraic property of interpolating between
the Fourier symbol of~$-\Delta$ for small frequencies
and that of~$(-\Delta)^s$ for high frequencies --- see~\cite[Theorem~1.1]{MV} for details.
Then, the energy functional considered in~\cite{MV} on
a compactly supported function~$
u$ with values in~$[0,1]$
takes the form
\begin{equation}\label{EFAN} {\mathcal{P}}_\e(u):=\e^{2s}\,\int_{\R^n} S_s(\xi)
\,|\widehat u(\xi)|^2\,d\xi+\int_{\R^n}W(u(x))\,dx,\end{equation}
where~$\widehat u$ is the Fourier transform of~$u$, and~$W$ is a nonnegative
double-well function vanishing at $0$ and $1$.
We observe that the scaling in~\eqref{EFAN} is reminiscent of the one in~\eqref{ENE}.
Then, recalling the scaling factors in~\eqref{local_part} and~\eqref{nonlocal_part},
one defines
$$ 
\mathcal{Q}_\e(u):=
\begin{cases}
\begin{aligned}
&\e^{-2s}\,\mathcal{P}_\e(u)
\quad &\text{if}\,\,\,s\in(0,1/2),
\\
&\abs{\e\log\e}^{-1}\,\mathcal{P}_\e(u)
\quad &\text{if}\,\,\,s=1/2,
\\
&\e^{-1}\,\mathcal{P}_\e(u)
\quad &\text{if}\,\,\,s\in(1/2,1).
\end{aligned}
\end{cases}$$
As proved in~\cite{MV},
when $s\in[1/2,1)$, the $\Gamma$-limit of the functional~$\mathcal{Q}_\e$ 
turns out to be the classical perimeter (up to normalizing constants), in analogy with~\cite{SV}. On the other
hand, when $s\in(0,1/2)$, the $\Gamma$-limit of~$\mathcal{Q}_\e$
is a new nonlocal energy functional, structurally different from the fractional
Laplacian
and from the ones that have been investigated in the literature, given by
$$ \mathcal{Q}(u):=
\begin{cases}\displaystyle
\int_{\R^n} S_s(\xi)
\,|\widehat u(\xi)|^2\,d\xi\quad &\text{if $u = \chi_E$, for some~$E\subset\R^n$,}\\
0\quad &\text{otherwise}.
\end{cases}$$
We refer to~\cite[Theorem~1.3]{MV} for a precise statement about the $\Gamma$-convergence of~$\mathcal{Q}_\e$.
\medskip

In the context of nonlocal models for the phase separations of a
fluid in a container, we also mention the articles~\cite{AB,ABCP}, in which the authors study the $\Gamma$-convergence of an interaction energy with a summable kernel. In this case, the functional has a singularity which is weaker than the one in~\cite{SV}, and other techniques, different from the ones in~\cite{SV}, are used.
\medskip

We also mention that an analogue
of Conjecture~\ref{CDG} for the fractional Allen-Cahn equation $(-\Delta)^s u(x) = W'(u(x))$ opens an interesting line of research.
For this, we refer to the recent surveys~\cite{CW, PORT, DMV}.

\section{Limits in the fractional parameter~$s$}\label{JA:AKMZ09wjjjdjd}

Till now, our main focus in this note was on the limit behavior of
phase transition energy functionals for the rescaling parameter $\e$ going to zero and for a fixed nonlocal exponent~$s$.
However, it is also possible to consider limits in the fractional parameter~$s$. 
The first result that we present in this setting is a
``pointwise'' limit, for $s\to(1/2)^-$, of the interior and the exterior
contributions in the fractional perimeter of a set $E$ inside the container~$\Omega$,
that converge, respectively, to the perimeter of $E$ inside $\Omega$,
and to the perimeter of~$E$ on the boundary of $\Omega$.
Recalling the notation in~\eqref{iwdjc:01} and~\eqref{iwdjc:02}, we state it in the following theorem.

\begin{theorem}[Lombardini~\cite{L}, Maggi and Valdinoci~\cite{MaggiV}]
	\label{thm_pointwise}
	Let $s\in(0,1/2)$, $\Omega'\subset\R^n$ be an open set, and $E\subset\R^n$ with locally finite perimeter in $\Omega'$. 
	
	Then, for every open set $\Omega$ compactly contained in $\Omega'$ and with Lipschitz boundary, it holds that
	\begin{equation*}
	\begin{split}
	\lim_{s\to(1/2)^-}\left(\frac12-s\right)\textnormal{Per}_s^{int}(E,\Omega)&=\omega_{n-1}\textnormal{Per}(E,\Omega),
	\\
	\lim_{s\to(1/2)^-}\left(\frac12-s\right)\textnormal{Per}_s^{ext}(E,\Omega)&=\omega_{n-1}\mathcal{H}^{n-1}\left(\partial^* E\cap\partial\Omega\right),
	\end{split}
	\end{equation*}
	where $\omega_{n-1}$ is the measure of the $(n-1)$-dimensional unit ball and $\partial^*E$ is the reduced boundary of $E$.
\end{theorem}

The study of the pointwise limit of the $s$-perimeter addressed in Theorem~\ref{thm_pointwise} has its foundations in the results about the limit as $s\to(1/2)^-$ of the $W^{2s,1}$-seminorm of a function. This study
was initiated by Bourgain, Brezis, and Mironescu~\cite{BBM} (see also~\cite{Davila} for optimal assumptions), establishing
that the $W^{2s,1}$-seminorm
	\[
	\abs{u}_{W^{2s,1}(\Omega)}:=\iint_{\Omega\times\Omega }\frac{\abs{u(x)-u(y)}}{\abs{x-y}^{n+2s}}\,dx\,dy
	\]
rescaled by $(1/2-s)$ converges as $s\to(1/2)^-$ to the $L^1$-norm of $\nabla u$. Some further results in this direction are obtained in~\cite{P}, also establishing the $\Gamma$-convergence of the $W^{2s,1}$-seminorm to its pointwise limit. We point out that in~\cite{BBM,Davila,P} the authors consider exponents $1\leq p<+\infty$ and more general kernels than $\abs{x-y}^{-n-2s}$, studying integrals of the type \[\iint_{\Omega\times\Omega }\frac{\abs{u(x)-u(y)}^p}{\abs{x-y}^p}\rho_i(x-y)\,dx\,dy,\] 
where $\rho_i$ is a sequence of radial mollifiers and the limit is taken for~$i\to\infty$.

We also mention the recent contributions~\cite{BN0,BN1,BN2} 
carrying out the study of both pointwise and $\Gamma$-limits	as $\delta\to0$
of a family of nonlocal and nonconvex functionals of the type
	\[
	\delta^p\iint_{\Omega\times\Omega }\frac{\varphi\left(\abs{u(x)-u(y)}/\delta\right)}{\abs{x-y}^{n+p}}\,dx\,dy,
	\]
	where $\varphi$ is a non-decreasing function satisfying some boundedness and growth assumption. 
		\medskip

A $\Gamma$-convergence counterpart of Theorem~\ref{thm_pointwise}
is provided by a result in~\cite{ADM}, which establishes the $\Gamma$-convergence
of the fractional perimeter to the classical perimeter, as the fractional
parameter $s$ converges to $1/2$:

\begin{theorem}[Ambrosio, De Philippis, and Martinazzi~\cite{ADM}]
	\label{thm_adm}
	Let $E\subset\R^n$ be a measurable set, and $\Omega$ compactly contained in $\R^n$ with Lipschitz boundary. Then,
	\begin{itemize}
		\item[(i)] for every sequences $s_i\to(1/2)^-$ and $E_i$ of measurable sets with $\chi_{E_i}\to\chi_E$ in~$L^1_\textnormal{loc}(\R^n)$, we have
		\[
		\liminf_{i\to\infty}\left(\frac12-s_i\right)\textnormal{Per}_{s_i}^{int}(E_i,\Omega)\geq\omega_{n-1}\textnormal{Per}(E,\Omega);
		\]
		\item[(ii)] for every sequence $s_i \to(1/2)^-$ there exists a sequence $E_i$ with $\chi_{E_i}\to\chi_E$ in~$L^1_\textnormal{loc}(\R^n)$, such that
		\[
		\limsup_{i\to\infty}\left(\frac12-s_i\right)\textnormal{Per}_{s_i}(E_i,\Omega)\leq\omega_{n-1}\textnormal{Per}(E,\Omega).
		\]
	\end{itemize}
\end{theorem}

We observe that the role played by interior and exterior contributions
in Theorem~\ref{thm_adm} is similar in some aspects to the
one in Theorem~\ref{thm_local_gamma}. Indeed, from Theorem~\ref{thm_adm} and
the fact that $\textnormal{Per}_{s}(E,\Omega)\geq\textnormal{Per}^{int}_{s}(E,\Omega)$, we immediately deduce that the interior contributions in the fractional perimeter $\textnormal{Per}_s^{int}$ and the full $s$-perimeter $\textnormal{Per}_s$ attain the same $\Gamma$-limit as $s$ converges to $1/2$. In this sense, the exterior contributions in the fractional perimeter, which are given by the term $\textnormal{Per}_s^{ext}$, do not contribute\footnote{
	The counterpart of this fact for long-range phase transition models was discussed in Remark~\ref{rmk_boundary}.} 
to the $\Gamma$-limit.
\medskip

We also stress that the boundary contributions
in the limit present significant differences
when the $\Gamma$-limit is replaced by the pointwise one,
as a close comparison between Theorems~\ref{thm_pointwise}
and~\ref{thm_adm} clearly shows.
Indeed,
if one considers the pointwise convergence for $s\to(1/2)^-$, as
done in Theorem~\ref{thm_pointwise}, then the interior and the exterior
contributions of the fractional perimeter converge, respectively, to the classical
perimeter of the set inside the container and to the measure of the part of
the boundary of the set $E$ that coincides with the boundary of
the container $\Omega$.

More specifically, from Theorem~\ref{thm_adm} it follows that, in the sense of Definition \ref{def_gamma},
\begin{equation}\label{7MISS1} \Gamma-\lim_{s\to(1/2)^-}\left(\frac12-s\right)\textnormal{Per}_s^{ext}(E,\Omega)=
0,\end{equation}
but from Theorem~\ref{thm_pointwise}
it holds that, for a given set~$E$ with Lipschitz boundary,
\begin{equation}\label{7MISS2} \lim_{s\to(1/2)^-}\left(\frac12-s\right)\textnormal{Per}_s^{ext}(E,\Omega)=
\omega_{n-1}\,\textnormal{Per}(E,\partial\Omega).
\end{equation}

Even if at a first glance
the ``mismatch'' between~\eqref{7MISS1}
and~\eqref{7MISS2} can be surprising, or a bit disturbing, several arguments suggest important
differences between the $\Gamma$-limit in~\eqref{7MISS1}
and the ``pointwise'' limit in~\eqref{7MISS2}.
First of all, the $\Gamma$-convergence dealt with
in our setting relies on the $L^1$-topology,
which is ``weak'' enough to allow the approximation
of every set $E$ with 
a sequence of sets~$E_k$ such that~$(\partial E_k)\cap(\partial\Omega)=\varnothing$.
This fact makes it possible to ``optimize'' the recovery sequence
in the~$\limsup$ inequality of the $\Gamma$-convergence
setting (recall in particular point~(ii)
in Definition~\ref{def_gamma})
in such a way to ``avoid additional boundary contributions''.

Another reason for the discrepancy between the limits in~\eqref{7MISS1} and~\eqref{7MISS2} lies in the ``variational nature'' of $\Gamma$-convergence with a fixed boundary datum. 
For this, the allowed variations for the related
minimization problem are taken with compact 
support inside the domain~$\Omega$. 
In this sense, 
the $\Gamma$-limit is typically not naturally endowed with
additional boundary contributions, which would be not
compatible with the notion of local minimizers of the limit problem.

\section*{Acknowledgments}

S.D. and E.V. are supported
by the Australian
Research Council Discovery Project DP170104880
``N.E.W. Nonlocal Equations at Work''.

S.D. is supported by the
DECRA Project DE180100957 ``PDEs, free boundaries
and applications''.

P.M. is supported by the MINECO grant MTM2017-84214-C2-1-P and
is part of the Catalan research group 2017SGR1392. 

	The authors are members of INdAM-GNAMPA.
	Part of this work was carried out on the occasion of a very
pleasant visit of the second author to the University of Western Australia, which we thank for the warm hospitality. This visit was partially supported by a Ferran Sunyer i Balaguer scholarship of the Institut d'Estudis Catalans, granted to the second author.

\end{document}